\documentclass[12pt,twoside,reqno]{amsart}
\usepackage{graphicx}
\usepackage{amsmath}
\usepackage{amssymb}
\usepackage{amsfonts}
\usepackage{multicol}
\usepackage[ansinew]{inputenc} 
\usepackage{amscd} 
\usepackage[mathscr]{eucal}
\usepackage{tikz}
\usetikzlibrary{automata,arrows,positioning,calc}

\usepackage{hyperref, enumerate}
 
\newcommand{\R}{\mathbb{R}}
\newcommand{\C}{\mathbb{C}}
\newcommand{\N}{\mathbb{N}}
\newcommand{\Z}{\mathbb{Z}}

\newcommand{\SL}{{\rm SL}}

\newcommand{\Mat}{{\rm Mat}}

\newcommand{\abs}[1]{\bigl| #1 \bigr|} 
\newcommand{\norm}[1]{\lVert#1\rVert} 
\newcommand{\normtwo}[1]{
{\left\vert\kern-0.25ex\left\vert\kern-0.25ex\left\vert #1 
    \right\vert\kern-0.25ex\right\vert\kern-0.25ex\right\vert} }

\newcommand{\ep}{\epsilon}

\newcommand{\Pp}{\mathbb{P}}

\newcommand\restr[2]{{
  \left.\kern-\nulldelimiterspace 
  #1 
  \vphantom{\big|} 
  \right|_{#2} 
  }}

\theoremstyle{plain}
\newtheorem{theorem}{Theorem}
\newtheorem{proposition}{Proposition}
\newtheorem{corollary}[proposition]{Corollary}
\newtheorem{lemma}{Lemma}

\numberwithin{equation}{section}

\theoremstyle{remark}
\newtheorem{remark}{Remark}[section]

\theoremstyle{definition}
\newtheorem{definition}{Definition}[section]

\newcommand{\Mscr}{\mathscr{M}}
\newcommand{\Ascr}{\mathscr{A}}
\newcommand{\Bscr}{\mathscr{B}}
\newcommand{\Cscr}{\mathscr{C}}
\newcommand{\Escr}{\mathscr{E}}

\title[A random cocycle]{ A random cocycle with non H\"older \\ Lyapunov exponent }

\newcommand{\Ell}{C}
\newcommand{\Hyp}{D}

\begin{document}

\newcommand{\one}{{\bf 1}}
\newcommand{\Hscr}{\mathscr{H}}
\newcommand{\EE}{\mathbb{E}}
\newcommand{\sgn}{{\rm sgn}}

\author[P. Duarte]{Pedro Duarte}
\address{Departamento de Matem\'atica and CMAFCIO\\
Faculdade de Ci\^encias\\
Universidade de Lisboa\\
Portugal 
}
\email{pmduarte@fc.ul.pt}

\author[S. Klein]{Silvius Klein}
\address{Departamento de Matem\'atica, Pontif\'icia Universidade Cat\'olica do Rio de Janeiro (PUC-Rio), Brazil}
\email{silviusk@impa.br}

\author[M. Santos]{Manuel Santos}
\address{ Departamento de Matem\'atica, IST, Universidade de Lisboa, Portugal}
\email{manuel.batalha.dos.santos@ist.utl.pt}

\maketitle

\begin{abstract}
We provide an example of a Schr\"odinger cocycle over a mixing Markov shift for which the integrated density of states has a very weak modulus of continuity, close to the log-H\"older lower bound established by W. Craig and B. Simon in~\cite{CS83}.
This model is based upon a classical example due to Y. Kifer~\cite{Kifer-ex} of a random Bernoulli cocycle with zero Lyapunov exponents which is not strongly irreducible. 
It follows that the  Lyapunov exponent of a Bernoulli cocycle 
near this Kifer example cannot be H\"older or weak-H\"older continuous, thus providing a  limitation on the modulus of continuity of the Lyapunov exponent of  random cocycles.      
\end{abstract}

\section{Introduction}

This paper is concerned with providing limitations on the modulus of continuity of the (maximal) Lyapunov exponent (LE) of random linear cocycles. By a random linear cocycle we understand the skew-product dynamical system defined by a Bernoulli or a Markov shift on the base and a locally constant linear fiber map. We fix the base dynamics and vary the fiber map relative to the uniform norm, thus continuity is with respect to the fiber map.

\smallskip

 Continuity of the LE in a generic setting (i.e. assuming irreducibility and contraction, which in particular imply simplicity of the LE) was first established by H. Furstenberg and Y. Kifer~\cite{FKi}. Recently, the genericity assumption was removed by C. Bocker-Neto and M. Viana~\cite{BV} (in the two-dimensional Bernoulli setting) and by E. Malheiro and M. Viana~\cite{Viana-M} (in a certain two-dimensional Markov setting). A higher dimensional version of the result in~\cite{BV} was announced by A. Avila, A. Eskin and M. Viana~\cite[Note 10.7]{Viana-book}. All of these results are not quantitative, i.e. they do not provide a modulus of continuity for the LE.
 
 The first quantitative result, namely H\"older continuity of the LE, was obtained by E. Le Page~\cite{LePage} in the {\em generic}, Bernoulli setting. This result refers to a one-parameter family of random linear cocycles; as such, it has been widely used in the theory of discrete, random, one-dimensional or strip  Schr\"odinger operators (which give rise to such one-parameter families of cocycles). Still in the generic setting, extensions of this result were obtained by P. Duarte and S. Klein~\cite{DK-book} and by A. Baraviera and P. Duarte~\cite{BD}. See also P. Duarte and S. Klein~\cite{DK-31CBM} for a simpler approach in the two-dimensional setting.   

\medskip

More recently, the first two authors of this paper considered the problem\footnote{Independently and with different methods, the same problem has also been studied by E. Y. Tall and M. Viana.}
 of obtaining a modulus of continuity of the LE for two-dimensional random Bernoulli  linear  
cocycles in the absence of any genericity assumption (see \cite{DK-RandomGL2}), but assuming the simplicity of the maximal LE.  Under this assumption, the results in \cite{DK-RandomGL2} establish local weak-H\"older continuity of the maximal LE in the most ``degenerate'' situation (i.e. in the vicinity of a diagonalizable cocycle) and local H\"older continuity elsewhere. There is work in progress establishing similar results  for Markov cocycles. 

\smallskip

A natural question arising from these developments is determining {\em how weak} the modulus of continuity of the LE can be. An example of B. Halperin, made rigorous by B. Simon and M. Taylor in~\cite{SimonTaylor}, shows that at this level of generality, the LE cannot be more regular than H\"older, and in fact the H\"older exponent may be arbitrarily close to zero.  When the LE is simple (that is, positive, in the $\SL_2 (\R)$ setting),  by \cite{DK-RandomGL2} it is at least weak-H\"older continuous. Can it be much weaker than this when the LE is {\em not simple}?

\medskip

Y. Kifer~\cite{Kifer-ex} considered the random Bernoulli cocycle $(C,D;p,1-p)$ generated by the matrices
\begin{equation}
\label{C & D def}
C:= \begin{bmatrix}
        0 & - 1\\
        1 & 0
    \end{bmatrix} 
    \qquad \text{ and } \qquad    
D := \begin{bmatrix}
        e & 0\\
        0 & e^{-1}
    \end{bmatrix} 
\end{equation}
with probabilities $(p, 1-p)$. A simple calculation shows that if $p > 0$ then the corresponding Lyapunov exponent is $0$, while when $p=0$, the Lyapunov exponent is $1$, thus implying the discontinuity of the Lyapunov exponent as a function of the {\em probability vector} $(p,1-p)$ at the boundary of the simplex.
In this work we provide an upper-bound  for the regularity of the LE as a function of the {\em matrices} at $(C,D;\frac{1}{2},\frac{1}{2})$.
A similar upper-bound should hold for any probability $0<p<1$.


\medskip

So far the only available method for proving limitations on the regularity of the LE for random cocycles is that of Halperin. This method in fact relies on the Thouless formula, which relates the LE to another quantity called the integrated density of states (IDS). The Thouless formula is only available for Schr\"odinger (and  Jacobi) cocycles, which is not the case with  Kifer's example. Our idea was then to embed Kifer's example into a family of Schr\"odinger cocycles (thus making the Thouless formula applicable) but over a finite type mixing Markov shift. 

\medskip

The example in this paper shows a huge breakdown on the regularity of the IDS (Theorem~\ref{main theorem}) which implies a similar breakdown for the LE in Kifer's example (Theorem~\ref{main corollary}). In this example the two assumptions of the classical result of Le Page (and of its extensions) fail, namely the cocycle is not strongly irreducible and it has zero Lyapunov exponent. 

Furthermore, Proposition~\ref{Lipschitz prop} shows that given a cocycle with zero LE, if it is strongly irreducible, then the LE must be pointwise Lipschitz at that cocycle. 

Therefore, in some sense it is the simultaneous  failing of the two assumptions that produces the break in regularity.

\bigskip

\section{Basic concepts}
\label{sec: basic concepts}

\subsection*{Linear cocycles}
Consider a probability space $(X,\mu)$ and an ergodic  measure preserving  transformation $T\colon X\to X$ on $(X,\mu)$.
An {\em  $\SL(2,\R)$-linear cocycle} over $T$  is any map $F_A:X\times \R^2  \to  X\times \R^2$  defined by a measurable function 
$A:X\to \SL(2,\R)$ through the expression 
$$F_A(x,v):=( T x, A(x)\, v ).$$
When the base map $T$ is fixed we identify  $F_A$ with $A$.

The forward iterates $F_A^n$  are given by
$F_A^n(x,v)=(T^n x, A^{(n)}(x) v)$,  where  $$
 A^{(n)}(x):= A(T^{n-1} x)\,\ldots \, A(T x)\,A(x)\quad (n\in\N) \;.$$
  
 The  Lyapunov exponent (LE) of 
 $F_A$ is defined as the $\mu$-almost sure limit
 $$ L(A):= \lim_{n\to+\infty}\frac{1}{n}\,\log \norm{A^{(n)}(x)} ,$$
whose existence follows by Furstenberg-Kesten's theorem~\cite{FK60}.

 \bigskip

\subsection*{Schr\"odinger operators and cocycles}
Let $T\colon X\to X$ be an ergodic transformation
over a probability space $(X,\mu)$. 
Denote by $l^2 (\Z)$ the usual Hilbert space of square summable sequences of real numbers $(\psi_n)_{n\in\Z}$. Note that
 $\lim_{n\rightarrow \pm \infty}\psi_n=0$ for all $\psi\in l^2(\Z)$.
 
Given some bounded  measurable function $\upsilon \colon X \to \R$,
at every site $n$ on the integer lattice $\Z$ we define the potential
$$v_n (x) := \upsilon (T^n x)\,.$$

The {\em discrete ergodic Schr\"odinger operator} with potential
$n\mapsto \upsilon_n(x)$  is the operator $H _x$ defined on $l^2 (\Z) \ni \psi  = \{\psi_n\}_{n \in \Z}$ as follows:
\begin{equation}\label{ s op}
[ H_x \, \psi ]_n := - (\psi_{n+1} + \psi_{n-1}) +  v_n (x) \, \psi_n\, .
\end{equation}

Due to the ergodicity of the system, the spectral properties of the family of operators $\{ H_x \colon x \in X \}$ are independent of $x$ $\mu$-almost surely.

Consider the Schr\"odinger  eigenvalue equation
\begin{equation}\label{s eq}
H_x \, \psi = E \, \psi\,,
\end{equation}
for some eigenvalue $E \in \R$ and  eigenvector $\psi = \{\psi_n\}_{n \in \Z}$.

The associated  {\em Schr\"odinger cocycle} is the cocycle $A_{ E}$ defined by
$$A_{  E} (x) := \left[ \begin{array}{ccc} 
 \upsilon (x)  - E  & &  -1  \\
1 & &  \phantom{-}0 \\  \end{array} \right] \in \SL_2 (\R)\,.$$

Note that the Schr\"odinger equation~\eqref{s eq} is a second order finite difference equation. An easy calculation shows that its formal solutions are given by 
\begin{align} 
\label{schrodinger formal solution}
\left[\begin{array}{c}
\psi_{n}\\
\psi_{n-1} \\ \end{array}\right]  =  
A^{(n)}_{ E} (x)
   \left[\begin{array}{c}
\psi_0\\
\psi_{-1} \\ \end{array}\right]\, .
\end{align}

\bigskip
 
Denote by $P_n\colon l^2(\Z)\to \C^{n+1}$ the coordinate projection to   $\{0,1, 2, \ldots, n\} \subset \Z$, by $P_n^\ast$ its adjoint and let 
\begin{equation}
\label{H truncation}
H^{(n)}_x  := P_n \, H (x) \, P_n^\ast .
\end{equation}

This finite rank operator is called the {\em $n$-truncation} of $H_x$.
  By ergodicity, the following limit exists for $\mu$-a.e. $x \in X$: 
$$N (E) := \lim_{n\to\infty} \, \frac{1}{n+1} \, \# \big( (- \infty, E] \cap \text{ Spectrum of } H^{(n)}_x \big)\,.$$
The function $E \mapsto N (E)$ is called the {\em integrated density of states} (IDS) of the family of ergodic operators $\{H_x \colon x \in X\}$ (see~\cite{David-survey}).

The LE and the IDS are related via the Thouless formula:
$$L (E) = \int_\R \log \abs{E-E'} d N (E')\,.$$


\subsection*{Random cocycles}
Let $\Sigma=\{1,\ldots, s\}$ be a finite alphabet, let $X=\Sigma^\Z$
be the compact product space of bi-infinite sequences of symbols in the set $\Sigma$ and let $T\colon X\to X$ be the {\em full shift map},
$T \{x_n\}_{n\in\Z} := \{x_{n+1}\}_{n\in\Z}$.  Given a probability vector  $q=(q_1,\ldots, q_s)$ on $\Sigma$, consider the product probability measure $\Pp_q=q^\Z$ on $X$. The map $T$ determines an ergodic transformation on $(X,\Pp_q)$ called the  two-sided {\em Bernoulli shift}.


Next we introduce the broader class of Markov shifts.
Recall that a {\em stochastic matrix} is any square matrix $P=(p_{ij})\in\Mat_s(\R)$ such that:
\begin{enumerate}
\item $p_{ij}\geq 0$ for all $i,j=1,\ldots, s$,
\item $\sum_{i=1}^s p_{ij}=1$ for all $j=1,\ldots, s$.
\end{enumerate}
A {\em $P$-stationary vector} is any probability vector $q\in \R^s$
such that $q=P\, q$, that is, 
$q_i=\sum_{j=1}^s p_{ij}\,q_j$ for all $i=1,\ldots, s$.
Each power $P^n$ is itself a  stochastic matrix.
Given a pair $(P,q)$ where $P$ is a stochastic matrix   and $q$ is a $P$-stationary probability vector   there exists a unique  probability measure $\Pp=\Pp_{(P,q)}$
on  $X=\Sigma^\Z$  such that the stochastic process $\{e_n\colon X\to \Sigma\}_{n\in\Z}$, $e_n(x):=x_n$,  has constant distribution $q$ and transition probability matrix $P$, i.e., for all $i,j=1,\ldots,s$,
\begin{enumerate}
\item $\Pp[ \, e_n=i \,] = q_i$, 
\item $\Pp[ \, e_{n}=i\,\vert \, e_{n-1}=j\, ] = p_{ij}$.
\end{enumerate}
The support of $\Pp_{(P,q)}$ is the space of admissible sequences
$$ B(P):=\{ x\in X\colon  p_{x_n x_{n-1}}>0 \, \forall n\in\Z \} $$
commonly referred to as the {\em sub-shift of finite type} defined by $P$.
The stochastic matrix $P$ is called {\em primitive} if  $P^n>0$ for some $n\geq 1$ (that is all the entries of $P^n$ are positive). If $P$ is primitive then the two-sided shift $T\colon X\to X$ is a mixing measure preserving transformation on $(X,\Pp_{(P,q)})$,  called a {\em mixing Markov shift}.


A  (locally constant)  {\em random cocycle} is any cocycle
$A:X \to\SL(2,\R)$ over a Bernoulli or Markov shift $T$  
such that  $A(\{x_n\})$ depends only on the first coordinate $x_0\in\Sigma$. Once the base dynamics given by the full shift is fixed, a  random  cocycle is completely determined by a list of $s$ matrices, $A_1,\ldots, A_s\in \SL(2,\R)$, such that
$A(\{x_n\})=A_{x_0}$.

\bigskip

\subsection*{Modulus of continuity}
 
Any continuous   and strictly-increasing function\\
 $\omega\colon [0,+\infty)\to [0, +\infty)$ with $\omega(0)=0$
 will be referred to as a {\em modulus of continuity}. Given a metric space $(X,d)$, we say that a function $f\colon X\to\R$ has modulus of continuity $\omega$ if  
 $$\abs{f(x)-f(y)}\leq \omega(d(x,y)),\quad \forall\, x,y\in X. $$
 
\smallskip

Let us recall some common moduli of continuity.

\smallskip
 
A function $f\colon X\to\R$ is  {\em  H\"older} continuous if it has  modulus of continuity $\omega(r)=C\,r^\alpha = C\, e^{ - \alpha\, \log \frac{1}{r}}$ for some pair of constants $C<\infty$ and $0<\alpha\leq 1$.
When $\alpha=1$ this corresponds to {\em Lipschitz} continuity.

\smallskip

A function $f$ is  {\em weak-H\"older} continuous if it has  modulus of continuity   $\omega(r)=  C\, e^{- \alpha\, (\log \frac{1}{r})^\theta}$ for some constants $C<\infty$, $0<\alpha\leq 1$  and $0<\theta\leq 1$.
When $\theta=1$, this corresponds to H\"older continuity.

\smallskip
 
A function $f$ is  {\em $\log$-H\"older} continuous if it has  modulus of continuity   $\omega(r)=  C \, (\log \frac{1}{r})^{-1} $ for some constant $C<\infty$.

\smallskip

Additionally, we define a stronger modulus of continuity than $\log$-H\"older.

\begin{definition}
We say that a function $f$ is $(\gamma, \beta)$-log-H\"older continuous if it has a modulus of continuity $\omega(r)=  C\, e^{- \beta \, (\log \, \log \frac{1}{r})^\gamma}$ for some constants $C<\infty$, $\gamma \geq 1$  and $\beta \ge 1$.
\end{definition}

Note that when $\gamma=1$ and $\beta=1$, this corresponds to $\log$-H\"older continuity, while as these parameters increase, the modulus of continuity improves.


Below we summarize the relations  between these moduli of continuity.
\begin{center}
H\"older \; $\Rightarrow$ \; weak-H\"older 
\; $\Rightarrow$ \; $(\gamma, \beta)$-log-H\"older \; $\Rightarrow$ \;  log-H\"older 
\end{center}

\bigskip

For discrete ergodic Schr\"odinger operators,
W. Craig and B. Simon~\cite{CS83} proved that the IDS is always log-H\"older continuous. 

M. Goldstein and W. Schlag~\cite[Lemma 10.3]{GS-Holder} showed that any singular integral operator on a space of functions preserves the modulus of continuity, as long as it is sharp enough. This applies to $(\gamma, \beta)$-log-H\"older continuity with $\gamma>1, \beta\ge1$ (and so to weak-H\"older and H\"older as well) but not to log-H\"older (or to a slighly stronger) modulus of continuity.

Since the Thouless formula $$L (E) = \int_\R \log \abs{E-E'} d N (E')$$ relates the IDS and the LE via such a singular integral operator (essentially the Hilbert transform), we conclude the following. 

\smallskip

\begin{proposition}
\label{MOC LE<=> IDS}
Given $\gamma>1, \beta\ge1$, the LE is $(\gamma, \beta)$-log-H\"older continuous if and only if the IDS is  $(\gamma, \beta)$-log-H\"older continuous.
\end{proposition}

However, the mere log-H\"older continuity of the IDS has no implications on the regularity of the LE (which in general may even be discontinuous).

\bigskip

\section{Main results}

Consider $\Sigma=\{0,a,b,c\}$ and the following Markov chain.

\begin{figure}
	\begin{center}
		\begin{tikzpicture}[->, >=stealth', auto, semithick, node distance=3cm]
		\tikzstyle{every state}=[fill=white,draw=black,thick,text=black,scale=1]
		\node[state]    (0)                     {$0$};
		\node[state]    (c)[above right of=0]   {$c$};
		\node[state]    (a)[below right of=0]   {$a$};
		\node[state]    (b)[below right of=c]   {$b$};
		\path
		(0) edge[loop left]     node{$\frac{1}{2}$}         (0)
		edge[bend right]    node{$\frac{1}{2}$}         (a)
		(c) edge[bend right]    node{$\frac{1}{2}$}         (0)
		edge                node{$\frac{1}{2}$}         (a)
		(b) edge[bend right]    node{$1$}                   (c)
		(a) edge[bend right]    node{$1$}                   (b);
		\end{tikzpicture}
	\end{center}
	\caption{Graph of the Markov chain on $\Sigma$.}
	\label{Sigma Graph}
\end{figure}
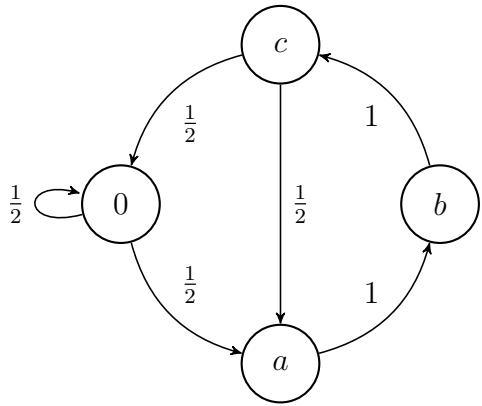

Let $X=\Sigma^{\Z}=\{0,a,b,c\}^{\Z}$, $T\colon X\rightarrow X$ be the shift map and let 
\begin{equation}
\label{prob q}
  q  = \left( \frac{1}{4},  \frac{1}{4},  \frac{1}{4},  \frac{1}{4}\right)    
\end{equation}
be a probability vector on $\Sigma$.
The transition probability  matrix  
\begin{equation}
\label{transition P}
 P =\begin{bmatrix}
\frac{1}{2} & 0 & 0 & \frac{1}{2}\\
\frac{1}{2} & 0 & 0 & \frac{1}{2}\\
0 & 1 & 0 & 0 \\
0 & 0 & 1 & 0
\end{bmatrix}
\end{equation}
is primitive since $P^5>0$ and the vector $q$ is 
  $P$-stationary. 
Therefore the pair $(P,q)$ determines a unique probability measure $\Pp=\Pp_{(P,q)}$ on $X$, and with this measure the map $T\colon X\to X$ 
is a mixing Markov shift.

\bigskip

Consider  now  the function  $\upsilon\colon \Sigma\rightarrow\R$ defined by
\begin{equation}
\label{Schrodinger Cocycle 1}
 \upsilon(a)=\upsilon(c)=-e,\; \upsilon(b)=-e^{-1} \;
\text{ and } \; \upsilon(0)=0. 
\end{equation}
This function determines the locally constant random potential   $\upsilon\colon X\rightarrow\R$ defined by
$\upsilon(x)=\upsilon(x_0)$ for all $x\in X$, which in turn determines the family of Schr\"odinger cocycles
\begin{equation}
\label{Schrodinger Cocycle 2}
A_{  E} (x) := \left[ \begin{array}{ccc} 
 \upsilon (x)  - E  & &  -1  \\
1 & &  \phantom{-}0 \\  \end{array} \right] \in \SL_2 (\R)  
\end{equation}
depending on the parameter $E\in\R$, over the Markov shift above.

Consider the corresponding discrete Schr\"odinger operator
$$[ H_x \, \psi ]_n := - (\psi_{n+1} + \psi_{n-1}) +  v (T^n x) \, \psi_n\, .$$

The following is the first main result of this paper.

\medskip

\begin{theorem}
\label{main theorem}
For any $\beta > 2$, the integrated density of states $N(E)$ of the discrete Schr\"odinger operator corresponding to the random Markov shift defined above is not  $(1, \beta)$-log-H\"older continuous  at $E=0$.
\end{theorem}

\bigskip

Recall that W. Craig B. Simon~\cite{CS83} established log-H\"older continuity of the IDS in the general setting of ergodic Schr\"odinger operators. W. Craig~\cite{Craig}, J. P\"oschel~\cite{Poschel} and more recently H. Kr\"uger and Z. Gan~\cite{Kruger-Zheng}  
constructed examples showing that this result is optimal in the setting of Schr\"odinger operators with limit periodic potentials. By a result of A. Avila~\cite{Avila-lp}, (non periodic) limit periodic potentials can be obtained  by sampling a continuous function along the orbits of a {\em minimal translation} of a Cantor group.   Finally, we were made aware of the work in progress~\cite{ALSZ} by A. Avila, Y. Last, M. Shamis and Q. Zhou, where log-H\"older is proven to be the optimal modulus of continuity for cocycles over a torus translation. 

\medskip

Theorem~\ref{main theorem}  shows that the log-H\"older continuity of the IDS obtained by Craig and Simon is nearly optimal at the other end of ergodic behavior, namely for Schr\"odinger operators with {\em random} potentials.    

\medskip

Using the above considerations regarding the transfer of a modulus of continuity via the Thouless formula, we derive the following about the regularity of the LE for random Bernoulli cocycles.

\begin{theorem}
\label{main corollary}
Consider the random Bernoulli cocycle~\eqref{C & D def} with probabilities $(\frac{1}{2},\frac{1}{2})$.
For any $\gamma>1$, the Lyapunov exponent is not $(\gamma, \beta)$-log-H\"older continuous  at this cocycle. In particular, it is not weak-H\"older or H\"older continuous.

\end{theorem}

\section{A Probabilistic Lemma }

The purpose of this section is to prove the following key lemma.

\begin{lemma} 
\label{CLT main lemma}
Consider the matrices $C$ and $D$ defined in~\eqref{C & D def}.
There exist $n_0\in\N$  such that for all $n\geq n_0$, the event that a random i.i.d. sequence $\{A_j\}_j$ 
with $\Pp[  A_j=\Hyp   ] = \Pp[  A_j=\Ell  ]=\frac{1}{2}$
satisfies  
$$A_n\,\cdots \, A_1\,A_0= \pm\,\begin{bmatrix}
e^\kappa & 0 \\ 0 & e^{-\kappa}
\end{bmatrix} \; \text{ with }\; 
  \kappa\geq \frac{1}{10}\,\sqrt{n}  $$
 has probability $> \frac{4}{10}$.
\end{lemma} 

This  will be proved at the end of this section.

\bigskip

\begin{proposition}
\label{product matrix lemma}
The matrices $C$ and $D$ satisfy
\begin{enumerate}
\item $\Ell^{-1}=-\Ell$ and  $\Ell^{2n}=\pm I$, $\Ell^{2n-1}=\pm \Ell$ for all $n\in\Z$.
\item $\Ell\,\Hyp\,\Ell=-\Hyp^{-1}$,
\item Any product $A_n\, \ldots\, A_1\, A_0$ with factors
$A_j\in \{\Hyp,\Ell\}$ has the form
	\begin{enumerate}
		\item $\pm\,\begin{bmatrix}
			 e^{\kappa} & 0 \\ 
			 0 & e^{-\kappa} 
			\end{bmatrix}$  when the number of factors $A_j=\Ell$ is even,
		\item $\pm\,\begin{bmatrix}
			  0 & -e^{\kappa} \\ 
			  e^{-\kappa} & 0
			\end{bmatrix}$   when the number of factors $A_j=\Ell$ is odd.
	\end{enumerate}
\end{enumerate}
\end{proposition}

\begin{proof}  
 $\Ell^2 = -I$
                 so, $\Ell^{2n}=\pm I$ depending on whether n is even or odd, which proves (1).  
 
Item (2)  follows from the fact that $\Hyp\Ell \Hyp \Ell=-I$. 
    
Let us prove (3). Given a product $A_n\, \ldots\, A_1\, A_0$  substitute any even length list of consecutive $\Ell$'s in it by a sign $\pm 1$, to get a product of the form:
\begin{equation}
    \label{eq 1}  \pm...\Ell\Hyp^m\Ell\Hyp^{m'}\Ell\Hyp^{m''}\Ell...
\end{equation}
  
Since there was an even number of $\Ell$'s cancellations,  the number of factors $A_j=\Ell$ has the same parity as the number of $\Ell$'s in~\eqref{eq 1}. Items (1) and (2) imply that $\Ell\,\Hyp^l = \Hyp^{-l}\,\Ell$ for all $l\in\Z$. 

Making use of these commutation relations combined with the identity $C^2=-I$, we can transform~\eqref{eq 1} into either $\pm \Hyp^\kappa$ or $\pm \Ell\,\Hyp^\kappa$, for some $\kappa\in\Z$. 
Since the number of  $\Ell$'s cancellations is even, the product~\eqref{eq 1} is  equal to $\pm \Hyp^\kappa$ if and only if the number of $\Ell$'s in it is even. This proves item (3).
\end{proof}

\bigskip

Consider a random i.i.d. process $\{A_n\}_{n\geq 0}$,
 such that  for all $n\geq 0$,
$$ \Pp[ \, A_n=\Hyp \, ] = \Pp[ \, A_n=\Ell \, ]=\frac{1}{2}.$$
Define the product process
\begin{align*}
 M_n  &:=  A_{n-1}\, \ldots\, A_1\, A_0  .
\end{align*}

By Proposition~\ref{product matrix lemma},
the process $M_n$ takes value in the union of the following two
 disjoint classes of matrices.
\begin{align*}
\Mscr_{+} &:=\left\{ \pm \begin{bmatrix}
e^\kappa & 0 \\ 0 & e^{-\kappa}
\end{bmatrix} \colon \kappa\in\Z \right\}\\
\Mscr_{-} &:=\left\{ \pm \begin{bmatrix}
0 & -e^{\kappa} \\ e^{-\kappa} & 0
\end{bmatrix} \colon \kappa\in\Z \right\}\\
\end{align*}
By the same proposition we have
$$ C\,\Mscr_{+}\subset \Mscr_{-},\quad C\,\Mscr_{-}\subset \Mscr_{+},\quad 
D\,\Mscr_{+}\subset \Mscr_{+} \; \text{ and }\;  D\,\Mscr_{-}\subset \Mscr_{-} .$$
Consider the sign valued process $\{\eta_n\}_n$
$$ \eta_n:= \left\{ 
\begin{array}{ccc}
+ &\text{ if } & M_n\in\Mscr_{+}\\
- &\text{ if } & M_n\in\Mscr_{-}\\
\end{array}
\right. $$
as well as the real valued process $\{S_n\}_n$ characterized by
$$ M_n = \left\{ \begin{array}{ccc}
 \pm \begin{bmatrix}
e^{S_n} & 0 \\ 0 & e^{-S_n} 
\end{bmatrix}  & \text{ if } & M_n\in\Mscr_{+} \\
\smallskip \\
 \pm \begin{bmatrix}
 0 & -e^{S_n} \\ e^{-S_n} & 0 
\end{bmatrix}  & \text{ if } & M_n\in\Mscr_{-}
\end{array}
\right.   \;.$$

\begin{proposition} \label{process rules}
The processes $\{\eta_n\}$ and $\{S_n\}$ relate to $\{A_n\}$ in the following way:
\begin{enumerate}
\item   $\eta_{n-1}=+$,\;  $A_n=C$  \; $\Rightarrow$\;  $\eta_{n}=-$,\; $S_n=S_{n-1}$,
\item   $\eta_{n-1}=-$,\; $A_n=C$  \; $\Rightarrow$\; $\eta_{n}=+$,\; $S_n=S_{n-1}$,
\item   $\eta_{n-1}=+$,\; $A_n=D$  \; $\Rightarrow$\; $\eta_{n}=+$,\; $S_n=S_{n-1}+1$,
\item   $\eta_{n-1}=-$,\; $A_n=D$  \; $\Rightarrow$\; $\eta_{n}=-$,\; $S_n=S_{n-1}-1$.
\end{enumerate}
\end{proposition}

\begin{proof}
Straightforward argument.
\end{proof}

Given a set of words $\Ascr\subset \{C,D\}^n$ and a letter $a\in \{C,D\}$ we define
$$ \Ascr\ast a=\{ (w,a)\in \{C,D\}^{n+1}\colon w\in\Ascr \} .$$
Each word $w=(w_0,w_1,\ldots, w_{n-1})\in\{C,D\}^n$ determines the cylinder
$$ C(w):=\{ x\in \{C,D\}^\Z \colon 
x_j=w_j,\; \forall\, j=0,1,\ldots, n-1 \} .$$
Throughout the paper we identify each set $\Ascr\subset \{C,D\}^n$  with the event $\cup_{w\in \mathscr{A}} C(w)$  determined by its words.
Because all words are equi-probable, the probability of $\Ascr$ (regarded as an event) is  $\Pp(\Ascr)=\frac{\#\Ascr}{2^{n} }$.

Define for each pair of integers $(n,i)$,
\begin{align*}
 \Ascr_+(n,i)&:=\{ w\in \{C,D\}^n \colon S_n(w)=i \, \text{ and } \, \eta_n(w)=+ \, \} \\
 \Ascr_-(n,i)&:=\{ w\in \{C,D\}^n \colon S_n(w)=i \, \text{ and } \, \eta_n(w)=- \, \}\\
 \Ascr(n,i)&:=\{ w\in \{C,D\}^n \colon S_n(w)=i    \, \}
\end{align*}
Note that $\Ascr(n,i) =  \Ascr_+(n,i) \cup  \Ascr_-(n,i)$.
Let us write $a_+(n,i):= \# \Ascr_+(n,i)$, $a_-(n,i):= \# \Ascr_-(n,i)$ and $a(n,i):= \# \Ascr(n,i)$.

\begin{proposition} \label{a(n,i) recurrence}
For any pair of  integers $(n,i)$,
\begin{align*}
a(n,i)=a_+(n,i)  & \quad \text{ if }\, n+i \, \text{ is even }\\
a(n,i)=a_-(n,i)  & \quad \text{ if }\, n+i \, \text{ is odd } .\\
\end{align*}
Moreover, the function $(n,i)\mapsto a(n,i)$ is  determined by the recursive relation
\begin{equation}
\label{binomial recursion}
a(n,i)=a(n-1,i-1) + a(n-1, i+1) 
\end{equation} 
and the initial conditions
\begin{equation}
\label{initial conditions}
 a(1,0)=a(1,1)=1 \, \text{ and } \,  a(1,i)=0 \, \text{ for all }\, i\neq 0,1 .
\end{equation} 
For all integers $(n,i)$
\begin{align}
\label{explicit 2i} 
 a(n,2i)=\binom{n-1}{\lfloor \frac{n-1}{2}\rfloor + i}\\
 \label{explicit 2i+1} 
 a(n,2i+1)=\binom{n-1}{\lfloor \frac{n}{2}\rfloor+i} \, .
\end{align}
\end{proposition}

\begin{proof}
Note that
$\Ascr(1,0)=\Ascr_-(1,0)=\{ (C)\}$ and  $\Ascr(1,1)=\Ascr_+(1,1)=\{ (D) \}$.
Hence $a(1,0)=a_-(1,0)=1$ and $a(1,1)=a_+(1,1)=1$, while $a(1,i)=a_+(1,i)=a_-(1,i)=0$ for all other $i$.
This proves~\eqref{initial conditions}.

With the notation introduced, from Proposition~\ref{process rules} we get
\begin{align}
\Ascr_+(n,i) = \Ascr_+(n-1,i-1)\ast D \, \sqcup \, \Ascr_-(n-1,i)\ast C  &\quad \text{ if } \, n+i \, \text{ is even }
\label{Ascr+}\\
\Ascr_-(n,i) = \Ascr_+(n-1,i)\ast C \, \sqcup \, \Ascr_-(n-1,i+1)\ast D  &\quad \text{ if } \, n+i \, \text{ is odd } 
\label{Ascr-}
\end{align}
where the symbol $\sqcup$ stands for disjoint union. These identities imply that
\begin{align}
a_+(n,i) = a_+(n-1,i-1) + a_-(n-1,i)  &\quad \text{ if } \, n+i \, \text{ is even }\\
a_-(n,i) = a_+(n-1,i) + a_-(n-1,i+1)   &\quad \text{ if } \, n+i \, \text{ is odd } 
\end{align}
 From the initial conditions we see by induction in $n$  that $a_-(n,i)=0$ and $a(n,i)=a_+(n,i)$  when $n+i$ is even, while
$a_+(n,i)=0$ and $a(n,i)=a_-(n,i)$  when $n+i$ is odd.

From~\eqref{Ascr+} and~\eqref{Ascr-} we get that
$$ a_-(n,i)=a_+(n-1,i)+a_-(n-1,i+1) = a_+(n,i+1)  $$
when $n+i$ is odd and
$a_-(n,i)=0=a_+(n,i+1)$ otherwise.
Therefore, because of these equalities, if $n+i$ is even then
\begin{align*}
a(n,i)  = a_+(n,i) &=
a_+(n-1,i-1) + a_-(n-1,i) \\
&= a_+(n-1,i-1) + a_+(n-1,i+1) \\
&= a(n-1,i-1) + a(n-1,i+1) .
\end{align*}
Similarly, if $n+i$ is odd then 
\begin{align*}
a(n,i)  = a_-(n,i) &=
a_+(n-1,i) + a_-(n-1,i+1) \\
&= a_-(n-1,i-1) + a_-(n-1,i+1) \\
&= a(n-1,i-1) + a(n-1,i+1) .
\end{align*}
This establishes identity~\eqref{binomial recursion}.
\begin{table}
\begin{tabular}{c|cccccccccccc}
	$i$ & $\cdots$ & $-4$ & $-3$ & $-2$ & $-1$ & $0$ & $+1$ & $+2$  & $+3$  & $+4$ & $+5$ & $\cdots$	\\
	\hline\\
	$a(1,i)$ &  & & & & & $1$ & $1$ & & & &\\
	$a(2,i)$  & & & & & $1$ & $1$ & $1$ & $1$  & & & \\
	$a(3,i)$  &   & & & $1$ & $1$ & $2$ & $2$ & $1$  &  $1$ & & \\
	$a(4,i)$  &   & & $1$ & $1$ & $3$ & $3$ & $3$ & $3$  &  $1$ & $1$ & \\
	$a(5,i)$  &   & $1$ & $1$ & $4$ & $4$ & $6$ & $6$ & $4$  &  $4$ & $1$ & $1$ & \\
	& & & &
\end{tabular}
\caption{Pascal's triangle for the numbers $a(n,i)$ }
\label{Pascal Triangle} 
\end{table}

Table~\ref{Pascal Triangle} presents the calculation of the first five rows of $a(n,i)$.
The recursive relations~\eqref{initial conditions} and~\eqref{binomial recursion} show that
both sequences $\{a(n,2i): -n+1 \leq 2 i\leq n  \}$ and 
$\{a(n,2i): -n+1 \leq 2 i+1\leq n  \}$ have exactly $n$ entries matching the binomial numbers 
$\{ \binom{n-1}{k}\colon 0\leq k\leq n-1 \}$.
Formulas~\eqref{explicit 2i} and~\eqref{explicit 2i+1} hold because as $2i$ ranges from $-n+1$ to $n$
the variable $k= \lfloor \frac{n-1}{2}\rfloor+i$ ranges from $0$ to $n$, while  as $2i+1$ ranges from $-n+1$ to $n$
the variable $k= \lfloor \frac{n}{2}\rfloor+i$ ranges from $0$ to $n$.
\end{proof}

\begin{proof}[Proof of Lemma~\ref{CLT main lemma}]
Consider the event $\Escr_n=[\, \eta_n=+,\, S_n\geq \frac{1}{10}\,\sqrt{n} ]$ whose probability we want to estimate,
which can be identified with the following set of words
$$ \Escr_n = \cup_{ i\geq \frac{1}{10}\,\sqrt{n}} \Ascr_+(n,i) =
\cup_{\substack{ i\geq \frac{1}{10}\,\sqrt{n} \\ n+i\,  \text{ even} }} \Ascr(n,i) . $$
By~\eqref{explicit 2i} and Proposition~\ref{a(n,i) recurrence}
\begin{align*}
\# \Escr_n = \sum_{2i\geq \frac{1}{10}\,\sqrt{n}} a(n,2i) &= \sum_{ i\geq  \frac{1}{20}\,\sqrt{n} + \lfloor \frac{n-1}{2}\rfloor } \binom{n-1}{ i}   \; \text{ when } \; n \; \text{ is even}.
\end{align*}
Similarly, by~\eqref{explicit 2i+1} and Proposition~\ref{a(n,i) recurrence}
\begin{align*}
\# \Escr_n =\sum_{2i+1 \geq \frac{1}{10}\,\sqrt{n}} a(n,2i+1) &= \sum_{i \geq \frac{1}{20}\,\sqrt{n}+\lfloor \frac{n}{2}\rfloor-\frac{1}{2}} \binom{n-1}{i} \; \text{ when } \; n \; \text{ is odd}.
\end{align*}
Therefore
\begin{equation}
\label{P(Escr_n)}
\Pp(\Escr_n)=\left\{ 
\begin{array}{ccc}
\frac{1}{2^{n-1}}\sum_{ i\geq  \frac{1}{20}\,\sqrt{n} + \lfloor \frac{n-1}{2}\rfloor } \binom{n-1}{ i} &\text{ if } & n \, \text{ is even }\\
\frac{1}{2^{n-1}}\sum_{i \geq \frac{1}{20}\,\sqrt{n}+\lfloor \frac{n}{2}\rfloor-\frac{1}{2}} \binom{n-1}{i}
&\text{ if } & n \, \text{ is odd }
\end{array}
\right. .
\end{equation} 

We now use the Central Limit Theorem (CLT) to estimate these sums.
Consider an i.i.d. process $\{Y_n\}$ where each $Y_n$ is a Bernoulli random variable with probabilities
$(\frac{1}{2},\frac{1}{2})$, that is $\Pp(Y_n=0)=\Pp(Y_n=1)=\frac{1}{2}$.
All moments of $Y_n$ are equal to $\frac{1}{2}$  and so is its standard deviation $\sigma(Y_n)=\frac{1}{2}$.

Next consider the normalized sum process
$$T_n=\frac{ Y_1+\ldots + Y_n -\frac{n}{2}}{\frac{\sqrt{n}}{2} } =
\frac{2\,( Y_1+\ldots + Y_n) - n}{\sqrt{n}} . $$
The CLT says that $T_n$ converges in distribution to the standard normal $N(0,1)$, whose
cumulative distribution is given by
$$ F(u):= \frac{1}{\sqrt{2\pi}}\, \int_{-\infty}^u  
e^{-\frac{x^2}{2}}\,dx  .
$$
More precisely this means that for all $u\in\R$,
$$ \lim_{n\to \infty} \Pp(T_n\leq u ) = F(u) . $$
On the other hand because the random variables $Y_n$ are Bernoulli,
\begin{align*}
\Pp(T_n\leq u) &= \Pp\left( \sum_{j=1}^n Y_j \leq \frac{n}{2}+u \frac{\sqrt{n}}{2}\right)
=\frac{1}{2^n}\sum_{i\leq \frac{n}{2}+u \frac{\sqrt{n}}{2}}\binom{n}{i} .
\end{align*}
Hence
\begin{equation}
\label{P(Tn-1>1/10)}
\frac{1}{2^{n-1} }\sum_{i> \frac{n-1}{2}+\frac{1} {20} \sqrt{n-1}}\binom{n-1}{i} =\Pp\left( T_{n-1} > \frac{1}{10} \right)
\end{equation}  
converges to $1-F(1/10)\approx 0.460172>0.4$. 
Comparing the sums in~\eqref{P(Escr_n)} and~\eqref{P(Tn-1>1/10)} we conclude that
$\lim_{n\to\infty} \Pp(\Escr_n)= 1-F(1/10)>0.4$.

The Berry-Esseen's Theorem (see~\cite{TaoBook}) implies that there exists   $C<\infty$ such that for all $n\in\N$
$$ \vert \Pp(T_n\leq u ) - F(u)\vert \leq \frac{C}{\sqrt{n}} .$$
 Using this fact, the threshold after which    $\Pp(\Escr_n)>0.4$ holds can be explicitly computed.
\end{proof}

\bigskip

\section{Proof of Theorem~\ref{main theorem}}

B. Halperin gave an example of a random Schr\"odinger cocycle where the IDS (hence also the LE), as a function of the energy $E$, cannot be better than H\"older continuous, with some explicitly given H\"older exponent. Our argument follows closely  the proof of this result given by B. Simon and  M. Taylor in~\cite{SimonTaylor}.
 

\begin{lemma}[Temple's Inequality]
\label{temple ineq}
Let $A$ be a self-adjoint operator in some Hilbert space.
Assume $\{f_j\}_{j=1}^k$ is an orthonormal family such that:
\begin{enumerate}
\item $\langle f_i, A f_j \rangle = \langle A f_i, A f_j \rangle = 0$ for all $i\neq j$ \text{ and}
\item $ \norm{ (A-E_0 ) f_j}\leq \varepsilon$,
\end{enumerate}
for some $\varepsilon>0$ and $E_0\in\R$.
Then $A$ has at least $k$ eigenvalues (counted with multiplicity)
in the range $[E_0-\varepsilon,E_0+\varepsilon]$.
\end{lemma}

\begin{proof}
See~\cite[Lemma A.3.2]{SimonTaylor}.
\end{proof}

\begin{definition}
\label{def quasi eigenfunction}
Any vector $f\in H$ such that $\norm{(A-E_0)\,f}\leq \varepsilon $ will be called an  $(\varepsilon,E _0)$-quasi-eigenfunction of the operator $A$.
\end{definition}

\bigskip

Consider now, throughout the rest of this section,
the family of Schr\"odinger cocycles defined in~\eqref{Schrodinger Cocycle 1} and~\eqref{Schrodinger Cocycle 2} over the Markov shift defined by~\eqref{prob q} and~\eqref{transition P}.

\begin{lemma}
\label{Ell and Hyp}
In our example the cocycle satisfies the following relations
\begin{eqnarray*}
    \Ell &=&\begin{bmatrix}
        \upsilon(0) & -1\\
        1            & 0
    \end{bmatrix}\\
    \Hyp &=&\begin{bmatrix}
        \upsilon(c) & -1\\
        1            & 0
    \end{bmatrix} \begin{bmatrix}
                        \upsilon(b) & -1\\
                        1            & 0
                  \end{bmatrix} \begin{bmatrix}
                                    \upsilon(a) & -1\\
                                    1            & 0
                                \end{bmatrix}
\end{eqnarray*}
where $C$ and $D$ are the matrices in~\eqref{C & D def}.
\end{lemma}

\begin{proof}
Straightforward calculation.
\end{proof}

A bi-infinite sequence $w=(\,\ldots, w_0, w_1,\ldots, w_n,\ldots\,)\in \Sigma^\Z$ is called  {\em allowable}  if it is a  path in the graph of Figure~\ref{Sigma Graph}. A typical allowable sequence looks like
$$\ldots0\,abc\,abc\,00\,abc\,000\, abc\, 0\ldots$$
Clearly the set of allowable sequence has full probability.

A word $w\in \Sigma^n$ is called \textit{allowable} if it is a path in the graph of Figure~\ref{Sigma Graph}.
We denote by $\Bscr(n)$ the set of all allowable words $w\in \Sigma^n$. A  word $w\in \Sigma^n$ is called \textit{admissible}  if
it is allowable and moreover it only contains full `$abc$' blocks. For instance the word $(00abc0ab)$ is allowable but not admissible.
We denote by $\Ascr(n)$ the set of all admissible words $w\in \Sigma^n$. We write $b(n)=\#\Bscr(n)$ and $a(n)=\# \Ascr(n)$.
Given $w\in X=\Sigma^\Z$ such that the finite word $(w_0,w_1,\ldots, w_{n-1})$
is admissible, we can decompose the iterate  $A_0^{(n)}(w)$ (at the energy level $E=0$) as a  product of matrices with  factors $C$ and $D$.

\smallskip

Given an admissible word $w=(w_0,w_1,\ldots, w_n)$ we define  
$$ M_w:= \begin{bmatrix}
\upsilon(w_n) & -1 \\1 & 0
\end{bmatrix}\,\cdots \,
\begin{bmatrix}
\upsilon(w_1) & -1 \\1 & 0
\end{bmatrix}\,
\begin{bmatrix}
\upsilon(w_0) & -1 \\1 & 0
\end{bmatrix} .  $$

\begin{proposition}
\label{quasi ef}
Let  $w^\ast = w_1\,0\, w_2=(a_0,a_1,\ldots, a_{n})$ be a finite admissible  word such that  for some integers $ k_1,  k_2 > K > 0$,
the sub-words
$w_1$ and $w_2$ satisfy
$$
M_{w_1}=\begin{bmatrix}
    e^{k_1} & 0\\
    0 & e^{-k_1}
\end{bmatrix} \quad \text{ and } \quad
 M_{w_2}=\begin{bmatrix}
    e^{k_2} & 0\\
    0 & e^{-k_2}
\end{bmatrix} \, .
$$

Consider any   $w\in X$ in the cylinder determined by $w^\ast$.
The associated n-truncation of the Schr\"odinger operator, defined in~\eqref{H truncation}, is given by:

$$H_w^{(n)}=\begin{bmatrix}
    \upsilon(a_0) & -1 & 0 &  &  &  & & 0\\
    -1 & \upsilon(a_1) & -1 & 0 &  & \ddots \\
    0 & -1 & \upsilon(a_2) & -1 & 0 &\\
     & 0 & -1 &  &   &\\
    & & 0 &  & \ddots & & &0\\ 
    &  \ddots  & & & & & & -1\\
    0 & & & & & 0& -1 & \upsilon(a_{n})
\end{bmatrix} .$$

Let $\psi = (\psi_j)_{j\in \Z}\in l^2(\Z)$ be the sequence defined recursively by 
$$ \begin{bmatrix}
    \psi_{0}\\
    \psi_{-1}
\end{bmatrix} = \begin{bmatrix}
    e^{-k_1}\\
    0
\end{bmatrix}\; \text{ and }\;
\begin{bmatrix}
    \psi_{j+1}\\
    \psi_j
\end{bmatrix} = \begin{bmatrix}
    \upsilon(a_j) & -1\\
    1 & 0
\end{bmatrix} \begin{bmatrix}
    \psi_j\\
    \psi_{j-1}
\end{bmatrix} \; \text{ for } \; 0 \leq j \leq n $$
with  $\psi_i=0$ for all $i<-1$ and $i>n+1$.
Then  
\begin{enumerate}
\item $\norm{\psi}\geq 1$,
\item $\psi$ is an $(\sqrt{2}\,e^{-K},0)$-quasi-eigenfunction of $H_w$, i.e., $ \norm{ H_w\, \psi} \leq \sqrt{2}\,e^{-K}$,

\item the truncation  $P_n \psi$ of $\psi$ to the range $[0,n]$
is also an $(\sqrt{2}\,e^{-K},0)$-quasi-eigenfunction of the truncated operator $H_w^{(n)}$.

\end{enumerate}

\end{proposition}

\begin{proof}
We only prove items (1) and (2). The proof of the third item is similar.

Consider $\psi=(\psi_j)_{j\in\Z}$ under the assumptions of the proposition. Then for all $0\leq j\leq n$, one has 
$\psi_{j+1} =\upsilon(a_j)\psi_j-\psi_{j-1}$, which implies that
$$(H_w\psi)_j = -(\psi_{j+1}+\psi_{j-1})+\upsilon(a_j)\psi_j = 0.$$\
In particular, since
$
\begin{bmatrix}
    \psi_{0}\\
    \psi_{-1}
\end{bmatrix} = \begin{bmatrix}
    e^{-k_1}\\
    0
\end{bmatrix}
$   we have

\begin{eqnarray*}
\begin{bmatrix}
    \psi_{n+1}\\
    \psi_{n}
\end{bmatrix} &=&\begin{bmatrix}
    e^{k_2} & 0\\ 0 & e^{k_2} 
    \end{bmatrix} \begin{bmatrix}
        0 & -1\\ 1 & 0
    \end{bmatrix} \begin{bmatrix}
        e^{k_1} & 0\\ 0 & e^{k_1}
    \end{bmatrix} \begin{bmatrix}
        e^{-k_1}\\ 0
    \end{bmatrix}
    =\begin{bmatrix}
    0\\
    e^{-k_2}
\end{bmatrix} .
\end{eqnarray*}
Thus, since $\psi_{-1}=\psi_{-2}=0$ we have
$$
(H_w\psi)_{-1} = -(\psi_{-2}+\psi_0) + \upsilon(w_{-1})\,\psi_{-1} = \psi_{0}= -e^{-k_1} .
$$
Similarly, since $\psi_{n+1}=\psi_{n+2}=0$ we have
$$
(H_w\psi)_{n+1} = -(\psi_{n+2}+\psi_{n}) + \upsilon(w_{n+1})\,\psi_{n+1} = \psi_{n}=  -e^{-k_2} .
$$
Finally, because $\psi_j=0$ for all $j\notin [0,n]$,
we get $(H_w\psi)_j=0$ for all $j\notin [-1,n+1]$.
Together, these bounds show that 
$$ \norm{ H_w\psi} = \sqrt{e^{-2 k_1}+e^{-2 k_2}}  \leq  \sqrt{2}\, e^{-K}  $$
and prove (2).

Denoting by $q$ the length of the word $w_1$ we have
\begin{eqnarray*}
\begin{bmatrix}
    \psi_{q}\\
    \psi_{q-1}
\end{bmatrix} &=&    \begin{bmatrix}
        e^{k_1} & 0\\ 0 & e^{k_1}
    \end{bmatrix} \begin{bmatrix}
        \psi_0\\ \psi_{-1}
    \end{bmatrix} = \begin{bmatrix}
        e^{k_1} & 0\\ 0 & e^{k_1}
    \end{bmatrix} \begin{bmatrix}
        e^{-k_1}\\ 0
    \end{bmatrix}
    = \begin{bmatrix}
    1\\
    0
\end{bmatrix}  
\end{eqnarray*}
which implies that $\norm{\psi}\geq \abs{\psi_q}= 1$, thus proving (1).
\end{proof}

\bigskip

Recall that $a(n)$ and $b(n)$ count, respectively, the admissible and allowable words of length $n$.

\begin{proposition}
	The sequence $b(n)$ satisfies
	$$  b(1)=4,\, b(2)=6,\, b(3)=9,\quad b(n)=b(n-1)+b(n-3)\quad\forall\, n\geq 4. $$
	The sequence $a(n)$ satisfies
	$$  a(0)=a(1)=a(2)=1,\quad a(n)=a(n-1)+a(n-3)\quad\forall\, n\geq 3. $$
	Moreover $b(n)=a(n+4)$ for all $n\geq 1$ and
	$$ \lim_{n\to +\infty} \frac{a(n)}{b(n)} =\lambda^{-4}=0.216757\ldots , $$
	where $\lambda>1$ is the Pisot number root the polynomial equation $x^3=x^2+1$.
\end{proposition}

\begin{proof}
	Note that $\Ascr(0)=\{\emptyset\}$,
	$\Ascr(1)=\{(0)\}$ and $\Ascr(2)=\{(00)\}$, which implies that
	$a(0)=a(1)=a(2)=1$. 
	
	Now, for $n\geq 3$, the decomposition
	$$\Ascr(n)= \Ascr(n-1)\ast 0 \, \sqcup \, \Ascr(n-3)\ast (abc) $$
	shows that $a(n)=a(n-1)+a(n-3)$. The first values of  $a(n)$,
	known as the Narayana's Cows sequence (see~\cite{Sloane}),
	 are shown in Table~\ref{Narayana}. 

	\begin{table}
		\begin{tabular}{c|cccccccccc}
			$n$ & $0$ & $1$ & $2$ & $3$ & $4$ & $5$ & $6$ & $7$  & $8$ & $\cdots$ 	\\
			\hline
			$a(n)$ & $1$  & $1$ & $1$ & $2$ & $3$ & $4$ & $6$ & $9$ & $13$  \\
		\end{tabular}
		\caption{ Narayana's cows sequence $a(n)$ }
		\label{Narayana} 
	\end{table}

	For the second sequence, note that
	\begin{align*}
	\Bscr(1) &=\{(0),(a),(b),(c)\}\\
	\Bscr(2) &=\{(00),(0a),(ab),(bc),(c0),(ca)\}\\
	\Bscr(3) &=\{(000),(00a),(0ab),(abc),(bc0),(bca),(c00),(c0a),(cab)\} 
	\end{align*}
	which gives $b(1)=4$, $b(2)=6$ and  $b(3)=9$.
	
	Given an allowable word $w\in \Sigma^n$ of length $n$ define
	the augmented word $w'$ of length $n+3$ as
	$$w':=  \left\{
	\begin{array}{ccc}
	w00a &\text{ if } & w \text{ ends in } \, 0\\
	wbca &\text{ if } & w \text{ ends in } \, a\\
	wc0a &\text{ if } & w \text{ ends in } \, b\\
	w00a &\text{ if } & w \text{ ends in } \, c 
	\end{array}
	\right.$$
	and let $\Bscr'(n+3):=\{w'\colon w\in \Bscr(n)\}$.
	Similarly, given an allowable word $w\in \Sigma^n$ of length $n$ define
	the augmented word $w^\ast$ of length $n+1$ as
	$$w^\ast:=  \left\{
	\begin{array}{ccc}
	w0 &\text{ if } & w \text{ ends in } \, 0\\
	wb  &\text{ if } & w \text{ ends in } \, a\\
	wc  &\text{ if } & w \text{ ends in } \, b\\
	w0  &\text{ if } & w \text{ ends in } \, c 
	\end{array}
	\right.$$
	and let $\Bscr^\ast(n+1):=\{w^\ast\colon w\in \Bscr(n)\}$.
	Note that
	$$\# \Bscr(n)=\# \Bscr'(n+3)=\#\Bscr^\ast(n+1).$$
	Finally, since
	$$ \Bscr(n)= \Bscr'(n-3) \sqcup \Bscr^\ast(n-1) , $$
	the sequence $b(n)$ satisfies $b(n)=b(n-1)+b(n-3)$.
	
	Looking at the first terms of $a(n)$ in Table~\ref{Narayana}, since $a(n)$ and $b(n)$
	satisfy the same recursive relation, we get that $b(n)=a(n+4)$ for all $n\geq 1$.
	
	Now the characteristic equation of the linear recursive equation for $a(n)$ is the polynomial equation $-x^3+x^2+1=0$. 
	This polynomial has $3$ roots, the Pisot number $\lambda=1.46557\ldots $
	and two more complex roots $\sigma, \overline{\sigma}$ inside the unit circle.
	Hence there are constants $c_1\in\R$ and $c_2\in\C$  such that
	$$a(n)=c_1\,\lambda^n + c_2\,\sigma^n + \overline{c}_2\,\overline{\sigma}^n\qquad \forall \, n\in\N.$$
	Therefore
	$$ \lim_{n\to \infty}\frac{a(n)}{b(n)}
	= \lim_{n\to \infty}\frac{a(n)}{a(n+4)}
	= \lim_{n\to \infty}\frac{c_1\,\lambda^n + o(1) }{c_1\,\lambda^{n+4} + o(1)} =\frac{1}{\lambda^4} .$$
\end{proof}

\begin{corollary}\label{Pp(Ascr)>0.2617}
	$\lim_{n\to \infty} \Pp( \Ascr(n) ) =\frac{1}{\lambda^4} = 0.216757\ldots$.
\end{corollary}

\begin{proof}
	Because all transition probabilities in the graph of $\Sigma$ (see Figure~\ref{Sigma Graph}) have the same probability $\frac{1}{2}$,
	all allowable words $w\in \Bscr(w)$ are equi-probable. Hence
	$\Pp( \Ascr(n) ) =\frac{a(n)}{b(n)}$.
\end{proof}

\bigskip

\begin{proposition} \label{CLT main proposition}
Given a $\Sigma$-valued stationary Markov chain $\{w_n\}_{n\in\Z}$ with stochastic transition matrix~\eqref{transition P}, there exists $n_0\in\N$ such that for all $l\geq n_0$, the event that $(w_0, w_1,\ldots, w_l)$ is an admissible word and
$$ M_{(w_0, w_1,\ldots, w_l)} = \pm \,\begin{bmatrix}
e^{\kappa} & 0 \\ 0 & e^{-\kappa} 
\end{bmatrix} \; \text{ with }\; 
  \kappa\geq \frac{1}{10}\,\sqrt{\frac{l}{3}}  $$
  has probability $> \frac{8}{100}$.
\end{proposition}

\begin{proof}
Consider the set $\Ascr_{l+1}=\Ascr(l+1)$ of all admissible words in $\Sigma^{l+1}$.
By Corollary~\ref{Pp(Ascr)>0.2617}, if $n_0$ is large enough and $l\geq n_0$, $\Pp(\Ascr_{l+1})>0.216$.

By Proposition~\ref{product matrix lemma} we can define $\kappa\colon \Ascr_{l+1} \to \Z$ such that for all $w\in\Ascr_{l+1}$,
$$ M_w = \pm \begin{bmatrix}
e^{\kappa(w)} & 0 \\ 0 & e^{-\kappa(\omega)} 
\end{bmatrix}\; \text{ or }\;
M_w = \pm \begin{bmatrix}
0 & -e^{\kappa(w)}  \\  e^{-\kappa(\omega)} & 0 
\end{bmatrix} . $$

Define also the functions $\rho, N\colon \Ascr_{l+1} \to \N$,
where $\rho(w):=\#\{ 0\leq j \leq l\colon w_j=0\}$ 
and  $N(w):=\rho(w) + \frac{l-\rho(w)}{3}$  counts the number of $0$'s plus the number of $abc$ blocks in an admissible word $w$.

To finish we  now derive the lower bound for the probability of the word set
$$ \Bscr_l :=\left\{ w\in \Ascr_{l+1} \colon \kappa(w)  
\geq \frac{1}{10}\,\sqrt{l/3},	\; \rho(w)\, \text{ even }  \right\} . $$

Applying the Law of Total Probabilities   we have
\begin{align*}
\Pp(\Bscr_l) &= \sum_{n=l/3}^l \Pp[\Ascr_{l+1}\cap \{N=n\}]\,
\Pp\left[ \left. \kappa   
\geq \frac{1}{10}\,\sqrt{l/3},\, \rho\, \text{ even }  \; \right\rvert \, \Ascr_{l+1}\cap \{N=n\}\right]\\
&\geq  \sum_{n=l/3}^l \Pp[\Ascr_{l+1} \cap \{N=n\}]\,
\Pp\left[ \left. \kappa   
\geq \frac{1}{10}\,\sqrt{n},\, \rho\, \text{ even }  \; \right\rvert \, \Ascr_{l+1} \cap \{N=n\}\right]\\
& >  \sum_{n=l/3}^l \Pp[\Ascr_{l+1} \cap \{N=n\}]\,
\frac{4}{10} =\frac{4}{10}\,\Pp(\Ascr_{l+1})> 0.4\times 0.216> 0.08 .
\end{align*}

In the first step we have used that $n\geq l/3$.
Also, by Lemma~\ref{CLT main lemma} we get
$$ \Pp\left[ \left. \kappa   
\geq \frac{1}{10}\,\sqrt{n}, \; \rho\, \text{ even }  \; \right\rvert \, \Ascr_{l+1}\cap \{N=n\}\right] > \frac{4}{10} . $$
This concludes the proof.
\end{proof}

\smallskip

Let $\Bscr_{l}$ be  the set of admissible words in $\Sigma^{l+1}$
such that
$ M_w= \pm\,\begin{bmatrix}
e^{\kappa} & 0 \\ 0 & e^{-\kappa}
\end{bmatrix}$  with $\kappa \geq \frac{1}{10}\,\sqrt{l/3}$.
By Proposition~\ref{CLT main proposition}, $\Pp(\Bscr_{l}) > 8/100$.
Next  consider the event $\Cscr_{l}=\Bscr_l\,0 \,\Bscr_l$ of all admissible words
$(w_1,0,w_2)\in\Sigma^{2l+3}$ with $w_1,w_2\in \Bscr_l$.

\smallskip

\begin{proposition}
$\Pp(\Cscr_l)=\Pp(\Bscr_l)^2>\frac{4}{625}$.
\end{proposition}

\begin{proof}
Let $\Bscr_l^-$ be the cylinder associated with admissible words $w=(w_0,\ldots , w_{2l+2})\in \Sigma^{2l+3}$
such that $(w_0,\ldots, w_l)\in \Bscr_l$. Analogously, let $\Bscr_l^+$ be the cylinder associated with admissible words 
$w=(w_0,\ldots , w_{2l+2})\in \Sigma^{2l+3}$
such that $(w_{l+2},\ldots, w_{2l+2})\in \Bscr_l$.
Note that if an admissible word $w$ lies in $\Bscr_l^-\cap \Bscr_l^+$ then its middle letter can only be `$0$',
i.e. $w_{l+1}=0$. Therefore $\Cscr_l=\Bscr_l^-\cap \Bscr_l^+$.

Given an admissible word $w=(w_0,\ldots, w_{n-1})$ of length $n$,
define $$P(w)=\prod_{i=1}^{n-1} P_{w_{i}, w_{i-1}}$$
where $P_{i,j}$ stands for the transition probability from state $j$ to state $i$.

Because the transition probabilities in $\Sigma$ are all equal to $1/2$ (see Figure~\ref{Sigma Graph}),
$P(w)=\frac{1}{2^{n-1}}$ for any admissible word $w$ of length $n$. The cylinder $C(w)$ determined by this word has probability
$\Pp(C(w))=\frac{1}{4}\,P(w)=\frac{1}{2^{n+1}}$ because the stationary probability on $\Sigma$ is
$q=(\frac{1}{4},\frac{1}{4},\frac{1}{4},\frac{1}{4})$. Hence
$$\Pp(\Bscr_l)= \sum_{w\in\Bscr_l} \frac{1}{4}\, P(w) . $$
Given an admissible word $w=(w,0,w')\in \Cscr_l$, since
$$ P(w,0,w') = P(w)\,\frac{1}{2}\,\frac{1}{2}\, P(w') $$
we have
\begin{align*}
\Pp(\Cscr_l) &= \sum_{w,w'\in\Bscr_l} \Pp(C_{(w,0,w')}) =  \sum_{w,w'\in\Bscr_l} \frac{1}{4}\,P(w,0,w') \\
&= \sum_{w,w'\in\Bscr_l} \frac{1}{4}\,P(w)\,\frac{1}{4}\,P(w') 
=\left( \sum_{w\in\Bscr_l} \frac{1}{4}\,P(w)\right)^2=\Pp(\Bscr_l)^2.
\end{align*}
The inequality in the proposition' statement follows from Corollary~\ref{Pp(Ascr)>0.2617}.
\end{proof}

 \bigskip

Fix now positive (large) integers $l$, $m$
and set $L=m\,(2l+3)$. Break $[0,L]$ into $m$ equal blocks
$I_1= [0,2l+2]$, $I_2=[2l+3,4l+5]$, $I_3=[4l+5,6l+7]$, etc., with the last block being
$I_m=[2(m-1)l+3(m-1),2 m l+3 m-1]$.  We refer to 
$$I_j:= [2(j-1) l+3(j-1), \, 2j l+3 j-1]$$
as the {\em $j$-th  block}  of the word $w$. 

Moreover, the block of length $2l+1$,
$$I_j^\circ := [2(j-1) l+3(j-1)+1, \, 2j l+3 j-2] ,$$
 obtained by removing the first and  last symbols from  the $j$-th  block of $w$, is called the 
 {\em inner $j$-th block} of $w$.

\begin{lemma}
\label{counting evs}
Take positive integers $l$, $m$ and $L$ as above.
Consider  an admissible word $w\in X$ and define
$$n_{l,m}(w):=\# \left\{\, 1\leq j\leq m \,\colon \, 
\text{ the inner } j\text{-th block   of } \,  w 
\, \text{ lies  in }\, \Cscr_{l}  \, \right\} .$$
Then the $L$-truncation operator $H_w^{(L)}$ has at least $n_{l,m}(w)$ eigenvalues
(counted with multiplicities) in the range $[ - \sqrt{2}\,e^{-K_l},   \sqrt{2}\,e^{-K_l} ]$, with
$K_l:=\frac{1}{10}\,\sqrt{l/3}$.
\end{lemma}

\begin{proof}
Given a word $w\in X$, for each $1\leq j\leq m$ such that
the inner $j$-th block of $w$ lies in $\Cscr_l$ we take the quasi-eigenfunction  of Proposition~\ref{quasi ef} and shift it to become supported on $I_j^\circ$.  Let $f_1, f_2, \ldots ,f_{n_{l,m}(w)} \in l^2(\Z)$ be the list of functions thus obtained.
Since each $f_i$ vanishes outside some $I_{j_i}^\circ$, by
Proposition~\ref{quasi ef}  the truncated function $P_L f_i$ satisfies
$\norm{ H_w^{(L)} (P_L f_i) }\leq \sqrt{2}\,e^{-K_l}$.
Hence each $P_L f_i$ is a  $(\sqrt{2}\,e^{-K_l}, 0)$-quasi-eigenfunction 
of the truncated operator $H_w^{(L)}$.

By construction $P_L f_i$ vanishes at the endpoints of the block  $I_{j_i}$. It follows that $H_w^{(L)} (P_L f_i)$ is also supported on the block $I_{j_i}$. Because these blocks are pairwise disjoint, assumptions (2) of Lemma~\ref{temple ineq} are automatically satisfied. The conclusion follows then by Temple's inequality.
\end{proof}


\begin{proof}[Proof of Theorem~\ref{main theorem}]
Consider a typical admissible word $w=(w_n)_{n\in\Z}\in X$.
By ergodicity,
$$\lim_{m\rightarrow \infty}\frac{n_{l,m} (w)}{m} = \Pp(\Cscr_l) \geq \frac{4}{625} \, .$$

Hence by Lemma~\ref{counting evs},
\begin{eqnarray*}
N\left(\sqrt{2}\,e^{- K_l} \right)-N\left(-\sqrt{2}\,e^{- K_l} \right) \geq \lim_{L\rightarrow\infty}\frac{n_{l,m} (w)}{L} =   \lim_{m\rightarrow\infty}\frac{n_{l,m} (w)}{m (2l+3)} 
 \geq \frac{4}{625}\,\frac{1}{2l+3} \, .
\end{eqnarray*}


Given  $\ep>0$, consider the $(2+\epsilon)$-log H\"older modulus of continuity
$$ \omega(r):= C\,\left( - \log r \right)^{-(2+\epsilon)} .$$

For sufficiently large $l$ 
by Lemma~\ref{counting evs} one has
$K_l \approx l^{1/2}$, so
$$\omega(e^{-K_l}) = C\, K_l^{-(2+\epsilon)}\lesssim l^{-\frac{2+\epsilon}{2}} \ll \frac{4}{625\,(2l+3)} .
  $$
Thus
$$N\left( \sqrt{2}\,e^{- K_l} \right)-N\left(- \sqrt{2}\,e^{- K_l} \right)\gg \omega(2\,\sqrt{2}\, e^{-K_l}) , $$
which means that the IDS is not $(2+\epsilon)$-log H\"older continuous.
 \end{proof}
 

\section{ Proof of Theorem~\ref{main corollary} }
\label{Theorem 2}

Recall that $\Sigma=\{0,a,b,c\}$, $X=\Sigma^\Z$ and $T:X\to X$ denotes the two-sided shift. Let $\Ascr=C(0)\cup C(a)$ be the union of cylinders determined by the one letter words `$0$' and `$a$'.
Let $N:\Ascr\to \N$ be the first return time to $\Ascr$ and
$T_\Ascr:\Ascr\to\Ascr$ be the induced (first return) map on $\Ascr$.
The function $N:\Ascr\to \N$ takes two values
$$ N(x)=\left\{ \begin{array}{lll}
1 & \text{ if } & x_0=0 \\
3 & \text{ if } & x_0=a 
\end{array}\right. $$
and hence the induced map on $\Ascr$ is given by
$$ T_\Ascr(x)=T^{N(x)} x = \left\{ \begin{array}{lll}
T x  & \text{ if } & x_0=0 \\
T^3 x & \text{ if } & x_0=a 
\end{array}\right. . $$
The family of Schr\"odinger cocycles $A_E:X\to \SL_2(\R)$
also induces a family of cocycles   $ \tilde{A}_E:\Ascr\to \SL_2(\R)$ over 
$T_\Ascr:\Ascr\to\Ascr$
defined by
$$ \tilde{A}_E(x)  = A^{(N(x))}_E(x) = \left\{ \begin{array}{lll}
A_E(x)   & \text{ if } & x_0=0 \\
A_E^{(3)}(x) & \text{ if } & x_0=a 
\end{array}\right. . $$

\begin{proposition} 
	For all $E\in\R$,
$ L(\tilde{A}_E) = \frac{3}{2}\, L(A_E)$.
\end{proposition}

\begin{proof}
The event $\Ascr$ has probability $\Pp(\Ascr)=\frac{2}{3}$. The induced measure $\Pp_\Ascr$ on $\Ascr$ is the conditional probability,
$\Pp_\Ascr(E):=\Pp[ E\vert \Ascr] =\frac{3}{2}\,\Pp(E\cap \Ascr)$. Hence the return time $N:\Ascr\to\N$ has expected value
$$ \int_{\Ascr} N\, d\Pp_\Ascr = \frac{1}{\Pp(\Ascr)}=\frac{3}{2} . $$
Consider now the sum process
$S_n:\Ascr\to \N$, defined by
$S_n (x): = \sum_{j=0}^{n-1} N(T_\Ascr^j x)$.
By the ergodicity of $(T,\Pp)$ and $(T_\Ascr,\Pp_\Ascr)$,
for $\Pp$-almost every $x\in\Ascr$,	
\begin{align*}
L(\tilde{A}_E) & = \lim_{m\to +\infty} \frac{1}{m}\,\log \norm{\tilde{A}_E^{(m)}(x) } = \lim_{m\to +\infty} \frac{1}{m}\,\log \norm{A_E^{(S_m(x))}(x) } \\
& = \lim_{m\to +\infty} \frac{S_m(x)}{m}\,  \lim_{m\to +\infty} \frac{1}{S_m(x)}\,\log \norm{A_E^{(S_m(x))}(x) } \\
&= \left( \int_{\Ascr} N\, d\Pp_\Ascr \right)\, L(A_E)
=\frac{3}{2}\, L(A_E) .
\end{align*}
This proves the proposition.
\end{proof}

Consider the map $h:\Ascr\to \{0,1\}^\Z$ that to each admissible sequence $x\in \Ascr$ associates the sequence $y\in \{0,1\}^\Z$ obtained from $x$ by replacing each block `$a b c$'
by the single letter `$1$'. This map conjugates the return map
$T_\Ascr:\Ascr\to \Ascr$ with the full Bernoulli shift
$T:\{0,1\}^\Z\to \{0,1\}^\Z$. It also determines a conjugation
between the family of cocycles $\tilde A_E$ over $T_\Ascr:\Ascr\to\Ascr$ and the family of random Bernoulli cocycles
$\hat A_E=(C(E),D(E))$ over the full shift
$T:\{0,1\}^\Z\to \{0,1\}^\Z$ defined by the following matrices
\begin{align*}
C(E) &=\begin{bmatrix}
\upsilon(0)-E & -1 \\ 1 & 0
\end{bmatrix} =\begin{bmatrix}
 -E & -1 \\ 1 & 0
\end{bmatrix} \\
D(E) &= \begin{bmatrix}
\upsilon(c)-E & -1 \\ 1 & 0
\end{bmatrix}
\begin{bmatrix}
\upsilon(b)-E & -1 \\ 1 & 0
\end{bmatrix}
\begin{bmatrix}
\upsilon(a)-E & -1 \\ 1 & 0
\end{bmatrix}\\
&= \begin{bmatrix}
-e-E & -1 \\ 1 & 0
\end{bmatrix}
\begin{bmatrix}
-\frac{1}{e}-E & -1 \\ 1 & 0
\end{bmatrix}
\begin{bmatrix}
-e-E & -1 \\ 1 & 0
\end{bmatrix}\\
&= \begin{bmatrix}
e- E\, p(E) & - E\,q(E)\\
E\,q(E) & \frac{1}{e} + E
\end{bmatrix}
\end{align*}
with $p(E):= E^2+ (2 e+ e^{-1})\, E + e^2$ and $q(E):= E+ e + e^{-1}$.
For $E=0$ we have $C(0)=C$ and $D(0)=D$ and hence $\hat A_0$ coincides with the Kifer example~\eqref{C & D def}.
The family $\hat A_E$ of random Bernoulli cocycles is analytic and has LE
$$ L(\hat A_E)= L( \tilde{A}_E ) = \frac{3}{2}\, L(A_E) . $$
Therefore, by Theorem~\ref{main theorem} and Proposition~\ref{MOC LE<=> IDS} the function
$E\mapsto L(\hat A_E)$ is not  $(\gamma, \beta)$-log-H\"older continuous  at $E=0$ for any $\gamma>1$.
This proves Theorem~\ref{main corollary}.

\section{ Irreducible cocycles }
\label{appendix II}

In this section we consider random $\SL_2$-cocycles over a  finite Bernoulli shift. Let $\Sigma=\{1,\ldots, s\}$ be a finite alphabet and fix some Bernoulli measure $\Pp_q=q^\Z$,
where $q$ is a probability vector on $\Sigma$.
Let  $T\colon X\to X$ denote the two sided shift on the space of sequences $X=\Sigma^\Z$ endowed with the probability measure $\Pp_q$.

Recall that a random cocycle over the Bernoulli shift $T$ is defined by a
locally constant measurable function $A\colon X\to \SL_2(\R)$,
i.e., a function which depends only on the $0$-th coordinate.
This implies that  $A$ is determined by a function
$A\colon \Sigma\to \SL_2(\R)$, or, in other words, by a list of $s$ matrices $A_1,\ldots, A_s\in\SL_2(\R)$.

\begin{definition}
A random cocycle $A\colon \Sigma\to \SL_2(\R)$ is said to be
irreducible if there is no point  $L\in \Pp(\R^2)$  such that   $A(x) L=L$ for all $x\in\Sigma$.
\end{definition}

\begin{definition}
A random cocycle $A\colon \Sigma\to \SL_2(\R)$ is said to be
strongly irreducible if there is no finite  subset $L\subset\Pp(\R^2)$, $L\neq \emptyset$, such that  for all $x\in\Sigma$, $A(x) L=L$.
\end{definition}

Clearly, strongly irreducible cocycles are also irreducible.
Irreducible cocycles which are not strongly irreducible will be
referred to as {\em simply irreducible cocycles}.

\begin{remark}
\label{rmk simp.irred}
Simply irreducible $\SL_2(\R)$-cocycles have zero Lyapunov exponents.
\end{remark}

The following statement is a classical theorem of H. Furstenberg~\cite{Fur}. 

\begin{theorem}
\label{thm Furstenberg}
Given  a random cocycle $A\colon \Sigma \to \SL_2(\R)$ assume:
\begin{enumerate}
\item[(a)] $A$ is strongly irreducible,
\item[(b)] the sub-semigroup generated by the  matrices $A_1,\ldots, A_s$ of the cocycle $A$ is not compact.
\end{enumerate}
Then $L(A)>0$.
\end{theorem}

The next proposition refers to the canonical $L^\infty$-norm: 
$$ \norm{A}_\infty:= \max_{x\in\Sigma} \norm{A(x)} .$$

\begin{proposition}
\label{Lipschitz prop}
Let $A\colon \Sigma\to \SL_2(\R)$ be a strongly irreducible cocycle such that
$L(A)=0$. Then there exists $C=C(A)<\infty$ such that for any  other cocycle  $B\colon \Sigma\to \SL_2(\R)$,
$$ \abs{L(A)-L(B)}\leq C\,\norm{A-B}_\infty .$$
\end{proposition}

\begin{proof}
By Theorem~\ref{thm Furstenberg}, the sub-semigroup $S\subset \SL_2(\R)$ generated by the  matrices $A_1,\ldots, A_s$ of the cocycle $A$ must be compact. This implies that the group generated by $S$ is also compact, and  that all matrices $A_j$ are orthogonal w.r.t.
some inner product. Denoting by $\norm{\cdot}'$ the corresponding operator norm (on the space of matrices) one has $\norm{A_j}'=1$ for all $j=1,\ldots, s$.

For any cocycle $B\colon \Sigma\to \SL_2(\R)$ we have
$$ 0\leq L(B) =   \lim_{n\to+\infty} \frac{1}{n}\,\EE_q[\,\log \norm{B^{(n)}}'] \leq  \EE_q[\,\log \norm{B}']  .$$
Note that $\EE_q[\log \norm{A}']=0$ for the cocycle $A$. Hence
\begin{align*}
\abs{L(B)-L(A)} &=L(B)\leq \EE_q[\log \norm{B}']\\
&\leq \EE_q[\, \log (\norm{A}' + \norm{B-A}') ]\\
&= \EE_q[\, \log \left( 1 + \norm{B-A}'\right) ]\\
&\leq  \EE_q[  \norm{B-A}'  ]\leq C\,\norm{A-B}_\infty
\end{align*}
for any constant $C<\infty$ such that $\norm{M}'\leq C\,\norm{M}$, for all $M\in \Mat_2(\R)$, where $\norm{\cdot}$ stands for the canonical operator norm on $\Mat_2(\R)$.
\end{proof}

\begin{remark}
\label{Lipschitz remark}
Proposition~\ref{Lipschitz prop} provides a modulus of Lipschitz continuity at $A$, but it does not imply that the LE is always Lipschitz in a neighborhood of $A$.
\end{remark}

\bigskip

We conclude this paper with the following question. 
Given a random $\SL_2(\R)$-cocycle under the assumptions of Proposition~\ref{Lipschitz prop}, is the LE always uniformly H\"older continuous in a neighborhood of that cocycle?

\subsection*{Acknowledgments}

The first author was supported  by Funda\c{c}\~{a}o para a Ci\^{e}ncia e a Tecnologia, under the projects: UID/MAT/04561/2013 and   PTDC/MAT-PUR/29126/2017.

The second author has been supported in part by the CNPq research grant 306369/2017-6 (Brazil) and by a research productivity grant from his  institution (PUC-Rio).

The third author was supported by a grant given by the Calouste Gulbenkien Foundation, under the project Programa Novos Talentos em
Matem\'{a}tica da Funda\c{c}\~{a}o Calouste Gulbenkian.
\bigskip


\providecommand{\bysame}{\leavevmode\hbox to3em{\hrulefill}\thinspace}
\providecommand{\MR}{\relax\ifhmode\unskip\space\fi MR }
\providecommand{\MRhref}[2]{%
  \href{http://www.ams.org/mathscinet-getitem?mr=#1}{#2}
}
\providecommand{\href}[2]{#2}

\end{document}